\PassOptionsToPackage{unicode}{hyperref}
\PassOptionsToPackage{hyphens}{url}
\documentclass[
]{article}
\usepackage{amsmath,amssymb}
\usepackage{lmodern}
\usepackage{ifxetex,ifluatex}
\ifnum 0\ifxetex 1\fi\ifluatex 1\fi=0 
  \usepackage[T1]{fontenc}
  \usepackage[utf8]{inputenc}
  \usepackage{textcomp} 
  \usepackage{newunicodechar}
            \newunicodechar{ρ}{\rho}
            \newunicodechar{∈}{\in}
            \newunicodechar{μ}{\mu}
            \newunicodechar{ξ}{\xi}
            \newunicodechar{Σ}{\Sigma}
            \newunicodechar{ℓ}{\ell}
            \newunicodechar{≠}{\neq}
            \newunicodechar{≅}{\cong}
            \newunicodechar{≤}{\leq}
            \newunicodechar{∣}{\mid}
            \newunicodechar{↦}{\mapsto}
            \newunicodechar{→}{\to}
            \newunicodechar{±}{\pm}
            \newunicodechar{⊆}{\subseteq}
            \newunicodechar{⊇}{\supseteq}
            \newunicodechar{⊂}{\subset}
	    \newunicodechar{⊤}{\top}
	    \newunicodechar{⟨}{\langle}
	    \newunicodechar{⟩}{\rangle}
\else 
  \usepackage{unicode-math}
  \defaultfontfeatures{Scale=MatchLowercase}
  \defaultfontfeatures[\rmfamily]{Ligatures=TeX,Scale=1}
\fi
\IfFileExists{upquote.sty}{\usepackage{upquote}}{}
\IfFileExists{microtype.sty}{
  \usepackage[]{microtype}
  \UseMicrotypeSet[protrusion]{basicmath} 
}{}
\makeatletter
\@ifundefined{KOMAClassName}{
  \IfFileExists{parskip.sty}{%
    \usepackage{parskip}
  }{
    \setlength{\parindent}{0pt}
    \setlength{\parskip}{6pt plus 2pt minus 1pt}}
}{
  \KOMAoptions{parskip=half}}
\makeatother
\usepackage{xcolor}
\IfFileExists{xurl.sty}{\usepackage{xurl}}{} 
\IfFileExists{bookmark.sty}{\usepackage{bookmark}}{\usepackage{hyperref}}
\hypersetup{
  hidelinks,
  pdfcreator={LaTeX via pandoc}}
\ifluatex
  \usepackage{selnolig}  
\fi

\usepackage{amsthm}
\newtheorem{theorem}{Theorem}
\newtheorem{lemma}{Lemma}

\newcommand{\Zeta}{Z}
\newcommand{\SL}{\operatorname{SL}}
\newcommand{\GL}{\operatorname{GL}}
\newcommand{\Irr}{\operatorname{Irr}}
\newcommand{\Ind}{\operatorname{Ind}}

\usepackage{authblk}
\title{Splitting fields of real irreducible representations of finite groups}
\author[1]{Dmitrii V. Pasechnik}
\affil[1]{Department of Computer Science\\
University of Oxford, UK\footnote{email: \url{dima@pasechnik.info}}
}
\date{2021-07-15}

\begin{document}
\maketitle
\begin{abstract}
We show that any irreducible representation $\rho$ of a finite group $G$
of exponent $n$, realisable over $\mathbb{R}$, is realisable over the field 
$E:=\mathbb{Q}(\zeta_n)\cap\mathbb{R}$ of real cyclotomic numbers of order $n$, and
describe an algorithmic procedure transforming a realisation of $\rho$ over  $\mathbb{Q}(\zeta_n)$
to one over $E$.
\end{abstract}

\section{Introduction} \label{introduction}

Let $G$ be a finite group of exponent $n$. A celebrated result by R.~Brauer
states that any complex irreducible character $\chi\in\Irr(G)$ of $G$
is afforded by an $F$-representation $\rho_\chi: G\to\GL_d(F)$, where $F=\mathbb{Q}(\zeta_n)$, the field of
cyclotomic numbers of order $n$ (here $\zeta_n:=e^{\frac{2\pi i}{n}}$), see \cite[(10.3)]{Is}.
Let $E:=\mathbb{Q}(\zeta_n)\cap\mathbb{R}\subset F$ be the maximal real subfield of $F$.
The first result of this note is as follows.
\begin{theorem}\label{mainthm}
Let $\chi$ be an irreducible real-valued character of $G$ of degree $d:=\chi(1)$
with Frobenius-Schur indicator $\nu_2(\chi)=1$.
Then $E$ is a splitting field of $\chi$,
i.e. $\chi$ is afforded by an $E$-representation $\rho$, and the Schur index $m_E(\chi)$ equals $1$.
\end{theorem}

Our proof of Theorem~\ref{mainthm} invokes Serre's induction theorem for real characters \cite{MR321908},
\cite[Theorem 73.18]{MR892316},
and then follows the line of proof of Brauer's theorem \cite[(10.3)]{Is}.
It is surprising that it has not appeared anywhere, at least as far as we know.

\paragraph{Remark.} Independently and simultaneously, Robert Guralnick and Gabriel Navarro proved Theorem~\ref{mainthm}
by a similar method, although not using  \cite{MR321908}.

Recall that the \emph{Frobenius-Schur indicator} $\nu_2(\chi):=\frac{1}{|G|}\sum_{g∈ G}\chi(g^2)$ is an invariant classifying complex representations of $G$ into
three different types, see \cite[(4.5)]{Is}.
Namely, $\nu_2(\chi)=0$ if $\chi$ is not real-valued, and $\nu_2(\chi)=-1$ if $\chi$ is real-valued, but is not afforded by a real-valued representation; $\nu_2(\chi)=1$ if
and only if $\chi$ is afforded by a real-valued representation. 

For a number field $K⊇ \mathbb{Q}$, the \emph{Schur index} $m_K(\chi)$ is an invariant of $\chi$
controlling the possibility to realise $\rho_\chi$ over $K$, see e.g. \cite[Sect.~41]{MR2215618} and \cite[Chapter~10]{Is}. Namely, let $S⊇ K$ be a splitting 
field of $\chi$. Then $m_K(\chi):=\min_{\substack{K⊆ M⊆ S\\ \text{$\rho_\chi$ realisable over $M$}}}  [M:K(\chi)]$, where we denoted by $[M:K(\chi)]$
the degree of $M$ as a field extension over $K(\chi)$, the field extension of $K$ generated by
the values of $\chi$. In particular, the claim of Theorem~\ref{mainthm} amounts to stating
that $m_E(\chi)=1$.

Apart from  theoretical significance, the question of finding a splitting field is relevant
in group theory algorithms. Standard algorithms such as J.~Dixon's algorithm 
\cite{MR1235797}
for constructing complex, and real, irreducible representations 
(one implementation in the computer algebra system GAP \cite{GAP4}
of it is described in \cite{MR2630014}) 
do induction from 1-dimensional
representations of subgroups of $G$, which are defined over $F$.
One advantage of working over $E$ instead is that the degree of $E$
is half of the degree of $F$.

In particular for applications, e.g. in extremal combinatorics, in physics, etc.
it is often necessary to reduce
a representation to a direct sum of real irreducibles, and exact
methods for this process benefit from explicit knowledge of the irreducibles,
using well known formulas from \cite[Sect.2.7]{MR0450380}, as implemented
in our GAP package RepnDecomp \cite{Hymabaccus2020}. 

Our second result amounts to the algorithmic counterpart of Theorem~\ref{mainthm}, that is,
 to a procedure to compute, for a representation
$\rho:G\to \GL_d(F)$
 realisable over reals,
an explicit matrix $Q∈ \GL_d(F)$ such that $Q^{-1}\rho(G)Q⊂ \GL_d(E)$,
i.e. $Q$ transforms $\rho$ to an $E$-representation.

\begin{theorem}\label{compthm}
	Let $\rho:G\to\GL_d(F)$ be a  representation of $G$ realisable over $\mathbb{R}$.
Then $P∈ \SL_d(F)$ such that $P\rho(g)=\overline{\rho(g)}P$ for any $g\in G$, and $P\overline{P}=I$, can be explicitly computed
from the $\rho(G)$-invariant forms. 
Let $ξ   ∈ F^*$ s.t.
$-\frac{\overline{ξ}}{ξ}$ is not an eigenvalue of $P$, and $Q:=\overline{ξ P}+ξ I$.
Then $Q∈ \GL_d(F)$ and
$Q^{-1}\rho(G)Q⊂ \GL_d(E)$.
\end{theorem}

The only part of Theorem~\ref{compthm} which uses Theorem~\ref{mainthm} is the claim
that $P$ can be chosen so that $P\overline{P}=I$. 
Algorithmically, one computes $P$ s.t.  $P\overline{P}=μ I$ for $0<μ∈ E$, 
and then has to solve the \emph{norm equation} 
\begin{equation}\label{normeq}
x\overline{x}=\mu,\qquad \text{for $x∈ F$.}
\end{equation}
Theorem~\ref{mainthm} implies that \eqref{normeq} is always solvable.
Several parts of the proof of Theorem~\ref{compthm} are contained
in \cite{MR1446124} and \cite{MR2531221}, 
although our approach is more explicit, and for odd $d$ we provide an explicit
solution (Lemma~\ref{odddeg}), not involving solving \eqref{normeq}, 
which is a nontrivial number-theoretic problem.

\section{Proof of Theorem~\ref{mainthm}}

Our main tool is Serre's induction theorem \cite[(73.18)]{MR892316}.
\begin{theorem}\label{serrethm}
  (Serre) The character $\chi$ of a real representation of $G$
  is a $\mathbb{Z}$-linear combination
  \begin{equation}\label{serreeq}
  \chi=\sum_{\phi}a_\phi  \Ind_H^G(\phi) 
  \end{equation} of real-valued induced characters
  $\Ind_H^G(\phi)$, with $H≤ G$, and $\phi$ a character of $H$.
  Further, $\phi$ is
  either linear
  and takes values $± 1$, or
  $\phi=\lambda+\overline{\lambda}$ for a linear character $\lambda$ of $H$, or 
  $\phi$ is dihedral. \qed
\end{theorem}

A \emph{diherdal character} $\phi$ of a group $H$ is a degree 2 irreducible character of $H$
s.t. $H/\ker\phi≅ D_{2m}$, dihedral group of order $2m$.

Note that by \cite[(10.2.f)]{Is} $m_E(\chi)$ divides $m_{\mathbb{Q}}(\chi)≤ 2$, where the latter inequality holds by the Brauer-Speiser Theorem \cite[p.171]{Is}.
Therefore it suffices to show that $m_E(\chi)=2$ is not possible in our situation.  

Let $\theta$ be a character of an $E$-representation of $G$. Then by \cite[(10.2.c)]{Is} $m_E(\chi)∣ [\theta,\chi]$.
Here $[,]$ is the usual scalar product of characters 
$[\theta,\chi]=\frac{1}{G}\sum_{g∈ G}\theta(g)\overline{\chi(g)}$, cf. \cite[(2.16)]{Is}.
As $\chi$ is irreducible, $[\chi,\chi]=1$, thus \eqref{serreeq} implies
\begin{equation}\label{serrescalar}
1=[\chi,\chi]=\sum_{\phi} a_\phi [\Ind_H^G(\phi),\chi]. 
\end{equation}
If every $\Ind_H^G(\phi)$ is an $E$-representation, then $m_E(\chi)=2$ is not possible, as otherwise an even integer on the right hand side of \eqref{serrescalar} equals 1.

It remains to see that every $\Ind_H^G(\phi)$ is an $E$-representation.

This is trivially
the case for linear $\phi$, and so we are left with the dihedral case and the case
$\phi=\lambda+\overline{\lambda}$. To simplify the rest of the proof, we use \cite[(10.9)]{Is}
which says that if a prime $p$ divides $m_E(\chi)$ then the Sylow $p$-subgroups of $G$
are not elementary abelian. For $p=2$ this means that $4∣ n$, i.e. $i:=\sqrt{-1}∈ F$. 

\begin{lemma}\label{Hchar}
  Let $H≤ G$, with $G$ of exponent $n$, $4∣ n$,
  and $\phi$ a character of $H$, either $\phi=\lambda+\overline{\lambda}$
  with $\lambda$ linear, or $\phi$ dihedral. Then $\phi$ is afforded by an $E$-representation.
\end{lemma}
\begin{proof}
Note that $E=\mathbb{Q}(\zeta_n+\zeta_{n}^{-1})$ and
$2\cos\frac{2\pi}{n}=\zeta_n+\zeta_{n}^{-1}$. As $4∣n$, it can be shown that 
$\sin\frac{2\pi}{n}∈ E$ in this case (in general this is not true).

In the case $\phi=\lambda+\overline{\lambda}$ we have $H/\ker\phi$ a cyclic group $C$ of
order $m$ dividing $n$, $C≅⟨ \zeta_m⟩$. We have
$\Zeta_m:=\begin{pmatrix} \cos\frac{2\pi}{m} & -\sin\frac{2\pi}{m} \\ \sin\frac{2\pi}{m} & \cos\frac{2\pi}{m}\end{pmatrix}∈\SL_2(E)$, and 
\begin{align*}
\rho_{\phi} : C & → \SL_2(E)\\
\zeta_m^k & ↦   \Zeta_m^k, \qquad 0≤ k<m
\end{align*}
is the desired $E$-representation of $C$ with character $\phi$.

For dihedral $\phi$ we have $H/\ker\phi$ a dihedral group $D=⟨ a,b∣ 1=a^m=b^2=(ab)^2⟩$ of
order $2m$ dividing $n$, with normal cyclic subgroup $C$ of order $m$, so that
the restriction $\phi_C=\lambda+\overline{\lambda}$ is as in the previous case, and
$\phi_{D-C}=0$. We have
$\Zeta_m ∈\SL_2(E)$ as in the previous case, and 
$R_0:= \begin{pmatrix} 1&0\\0&-1\end{pmatrix}∈ \GL_2(E)$ satisfying $R_0\Zeta_m R_0=\Zeta_m^{-1}$
and
\begin{align*}
\rho_{\phi} : D & → \GL_2(E)\\
a^k b^\ell & ↦   \Zeta_m^k R_0^\ell, \qquad 0≤ k<m,\ 0≤ℓ≤1,
\end{align*}
is the desired $E$-representation of $D$ with character $\phi$.
\end{proof}

This completes the proof of Theorem~\ref{mainthm}. The last step, i.e. the proof of
Lemma~\ref{Hchar}, could also
be accomplished in a less explicit way, by invoking the construction of Theorem~\ref{compthm};
the matrix $P$ mapping $\rho_\phi$ to its conjugate can be chosen to be equal to
$P=\begin{pmatrix} 0&1\\1&0\end{pmatrix}$, satisfying the only condition, $P\overline{P}=I$.
In particular this approach allows to prove a more general version of Lemma~\ref{Hchar} which 
does not require $4∣ n$.

\section{Proof of Theorem~\ref{compthm}}
The case $n=2$ is trivial, and we will assume $n\geq 3$ in what follows.

Recall that in general, $\chi$ has values in $F$,
while a real-valued character has values in $E$.
Whenever $\chi$ is $E$-valued, the image $\rho(G)$ of $G$ under
a representation $\rho:=\rho_\chi$ affording
$\chi$ leaves invariant a unique, up to scalar multiplication, non-zero $G$-invariant form
$M$. It is a classical result due to Frobenius and Schur that
if $M$ is symmetric then $\chi$ is afforded by a real representation $\rho$,  and $\nu_2(\chi)=1$, 
cf. \cite[(4.19)]{Is}.

Without loss of generality, $\chi(1)>1$. Indeed, if $\chi(1)=1$ then $\rho$ is the same as $\chi$, and we are done.

The proof
of Frobenius-Schur in \cite[(4.19)]{Is} starts with the elementary fact that if \(Q\) is a
transformation making \(\rho\) real then
\(Q^{-1}\rho Q=\overline{Q^{-1}\rho Q}\), thus
\(\overline{Q}Q^{-1}\rho=\overline{\rho Q}Q^{-1}\), and
\(P:=\overline{Q}Q^{-1}\) transforms \(\rho\) to \(\overline{\rho}\),
i.e.~\(P^{-1}\overline{\rho}P=\rho\).
Such a \(P∈ \GL_d(\mathbb{C})\) must exist irrespective of the existence of $Q$, as the
characters of \(\rho\) and \(\overline{\rho}\) are equal, although we
can give an explicit construction $P=\Sigma^{-1}M$, 
with $M$ as above, and $\Sigma$ the matrix of a positive definite Hermitian $\rho(G)$-invariant
form. 
\begin{lemma}\label{Pexists}
  Let $\chi$ be a real-valued character of $G$, and $\rho=\rho_\chi$ an $F$-representation
	affording $\chi$. Then
  $P:=\Sigma^{-1}M ∈ \GL_d(F)$ satisfies $P\rho(g)=\overline{\rho(g)}P$
  for all $g∈ G$.
\end{lemma}
\begin{proof}
As \(\chi\) is real, $\rho$ leaves invariant a nonzero \(G\)-invariant
bilinear form \(M\), i.e. \(g^\top Mg=M\) for all
\(g\in\rho\), cf.~e.g. \cite[(4.14)]{Is}.
As \(M\) can be found in the trivial sub-representation of the
tensor square of \(\rho\), \(M\in M_d({F})\). As well,
\(\det M\neq 0\), as the kernel of \(M\) would give rise to a
sub-representation of \(\rho\), contradicting irreducibility of
\(\rho\).

	Let \(\Sigma:=\sum_{h\in\rho(G)}h^\top \overline{h}\) - note that
\(\Sigma\) is a Hermitian positive definite matrix, in particular
	\(\det\Sigma> 0\), and $g^⊤ Σ\overline{g}=Σ$ for any $g∈ ρ(G)$.

Choose \(P:=\Sigma^{-1}M\). Let's check that
\(P^{-1}\overline{\rho}P=\rho\) (we use \(\det M\neq 0\) here). Let
	\(g\in\rho(G)\). Then, as
\((\overline{g}\Sigma^{-1}g^\top)^{-1}=(g^\top)^{-1}\Sigma \overline{g}^{-1}=\Sigma\),
\[\Sigma^{-1}Mg=\overline{g}\Sigma^{-1}g^\top Mg=\overline{g}\Sigma^{-1} M,\]
as required. 
\end{proof}
Now we have the equation 
\begin{equation}\label{PQeq}
PQ=\overline{Q}, \qquad \det Q\neq 0
\end{equation}
implying \(\overline{P}PQ=\overline{PQ}=Q\), i.e. \(\overline{P}P=I\).
The latter is an extra restriction, in the sense that our procedure does not
guarantee that $P$ computed as in Lemma~\ref{Pexists} satisfies \(\overline{P}P=I\).
In general, one will need to solve \eqref{normeq} and multiply $P$ by the
inverse of a solution. However, \eqref{normeq} will always be solvable by Theorem~\ref{mainthm}.

\begin{lemma}\label{PPbar}
  Let  $P∈ \GL_d(F)$ such that $Pg=\overline{g}P$
  for any $g∈\rho(G)$. Then $P\overline{P}=μ I$ for some $μ ∈ E$. 
\end{lemma}
\begin{proof}
Note that $\overline{Pg}=g\overline{P}$. Thus  $P\overline{P}\overline{g}=Pg\overline{P}=\overline{g}P\overline{P}$. Thus $P\overline{P}$
lies in the centraliser of an irreducible representation $\overline{\rho}$.
Hence, by Schur's Lemma, $P\overline{P}=μ I$, for some $\mu∈ F$.

It remains to show that $μ ∈ E$. Using Lemma~\ref{Pexists}, and recalling that
$\Sigma$ and $\Sigma^{-1}$ are Hermitian positive definite, i.e. $\Sigma^{-1}=U\overline{U}^⊤$, and $M=M^\top$, we have
$μ I=P\overline{P}=\Sigma^{-1}M\overline{\Sigma^{-1} M}$,
i.e. 
\[
  \mu\Sigma=M\overline{\Sigma^{-1}M}=
M\overline{U\overline{U}^⊤ M}=
M\overline{U}U^⊤\overline{M}=
(M\overline{U})(\overline{M\overline {U}})^⊤=
\overline{\mu\Sigma}^⊤
=\overline{\mu}\Sigma,
\] 
implying $\mu=\overline{\mu}$.
\end{proof}

It remains to solve \eqref{PQeq} so that \(Q\) has entries in the splitting field
of \(\rho\). Note that the solution of \eqref{PQeq} in
\cite[Ch.~4]{Is}
assumes that \(\rho\) is unitary; i.e.~\(\Sigma=I\); so in this case
\(P^\top=P\), and an explicit formula for \(Q\) is provided - which however
does not work for us, as it involves square roots of eigenvalues of
\(P\). Fortunately, in \cite[Prop.~1.3]{MR1446124}, there is an
algorithmic proof of existence of the required solution of \eqref{PQeq}. In
[loc.cit.] it is done for finite fields (and in bigger generality, for a
field automorphism \(\sigma\) of finite order, referring to this result
as a generalisation of \emph{Hilbert's Theorem 90}), and in
\cite{MR2531221}
it was noted that it works for number fields as well.
One can also find there an easier observation, that for a randomly chosen $Y∈ M_d(F)$
setting $Q=\overline{Y}+\overline{P}Y$ produces a solution to \eqref{PQeq} with
high probability. Here is an easy to prove variation of this claim.
\begin{lemma}\label{Qrnd}
  Let $P,Y∈ M_d(F)$ and $P\overline{P}=I$. Then $Q:=\overline{Y}+\overline{P}Y$
  satisfies $PQ=\overline{Q}$.
  Choosing $Y=ξ P$, with $ξ≠ 0$ and $-{ξ}/{\overline{ξ}}$ not being an eigenvalue of $\overline{P}$
  we have that $Q∈ M_d(F)$ satisfies \eqref{PQeq}.
\end{lemma}
\begin{proof}
  Note that $PQ=P\overline{Y}+P\overline{P}Y=Y+P\overline{Y}=\overline{Q}$, as claimed.
  The claimed choice of $ξ$ is possible as $F$ is dense in $\mathbb{C}$. 
  Further, with $Q=\overline{ξ}\overline{P}+ξ\overline{P}P=\overline{ξ}(\overline{P}+\frac{ξ}{\overline{ξ}}I)$ we see that $Qv=0$ holds for a non-zero vector $v$ if and only if
  $\overline{P}v=-\frac{ξ}{\overline{ξ}}v$, which is not possible by the choice of $ξ$.
\end{proof}

To complete the proof of Theorem~\ref{compthm} is suffices to observe that $Q^{-1}ρ(g) Q∈ M_d(E)$
for any $g∈ G$.

One can solve \eqref{normeq} in the case of odd $d$ without resorting to number-theoretic tools. 
\begin{lemma}\label{odddeg}
  Let $d=2k+1$. Then, \eqref{normeq} for $\mu$ in $P\overline{P}=μ I$ is solved by 
  $x=\mu^{-k}\det P$.
\end{lemma}
\begin{proof}
Let $\lambda:=\det P$. Then 
$\det (P\overline{P})=\lambda\overline{\lambda}=\overline{\lambda}\lambda=\det(μ I)=\mu^{2k+1}$.
Thus $\mu=\overline{\mu}=\frac{\lambda}{\mu^k}\frac{\overline{\lambda}}{\mu^k}$.
Replacing $P$ with $P'=\frac{\mu^k}{\lambda}P$ we see that $P'\overline{P'}=I$.
\end{proof}

\section{Related work and remarks}
The paper \cite{MR2531221} studies a closely related algorithmic question of
minimising the degree of the number field needed to write down a  complex
representation. It is known that such a field need not be cyclotomic.
On the other hand, computer algebra systems designed for computing in groups, such as 
GAP \cite{GAP4} and Magma \cite{MR1484478} typically use cyclotomic fields for computation
with characteristic zero representations of finite groups.
In particular, this work came as an analysis of a question \cite{GAPforum} posed on the
GAP discussion forum. 

Lemma~\ref{Pexists}  and its proof are essentially a refinement of an argument from the proof of
\cite[Thm.31]{MR0450380}. Lemmata~\ref{odddeg} and \ref{Qrnd} appear to be novel, as well as
Theorem~\ref{mainthm}.

\subsubsection{Acknowledgement.} The author thanks Denis Rosset, Stephen Glasby,
Kaashif Hymabaccus, Robert Guralnick, and many participants of MathOverflow,
in particular Will Sawin \cite{391269}, for fruitful discussions.

Special thanks go to Robert Guralnick, who, in his capacity
of a moderator for arXiv.org, caught a glaring omission in an earlier version of this text,
submitted there.

\bibliography{realcyc}
\bibliographystyle{alpha}
\end{document}